\documentclass[twoside]{article}

\usepackage{a4wide, amsmath,  amssymb, amscd, mathptmx, amsthm}
\usepackage{amsfonts}

\usepackage{color}

\usepackage{pstricks}
\usepackage[all]{xy}\CompileMatrices\SelectTips{cm}{12}

\theoremstyle{plain}
\newtheorem{Thm}{\sc Theorem}[section]
\newtheorem{theorem}[Thm]{\sc Theorem}
\newtheorem{corollary}[Thm]{\sc Corollary}
\newtheorem*{corollary*}{\sc Corollary}
\newtheorem{proposition}[Thm]{\sc Proposition}
\newtheorem*{proposition*}{\sc Proposition}
\newtheorem{lemma}[Thm]{\sc Lemma}

\theoremstyle{remark}
\newtheorem{remark}[Thm]{Remark}
\newtheorem{example}[Thm]{Example}
\newtheorem*{example*}{Example}
\newtheorem*{remark*}{Remark}

\newcommand{\cO}{{\mathcal O}}

\renewcommand{\AA}{{\mathbb A}}

\newcommand{\FF}{{\mathbb F}}

\newcommand{\NN}{{\mathbb N}}

\newcommand{\PP}{{\mathbb P}}
\newcommand{\QQ}{{\mathbb Q}}

\newcommand{\ZZ}{{\mathbb Z}}

\newcommand{\Fr}{{\mathop{\rm Fr}}}
\newcommand{\Spec}{\mathop{\rm Spec \, }}

\newcommand{\Hom}{{\mathop{{\rm Hom}}}}
\newcommand{\cHom}{{\mathop{{\cal H}om}}}

\renewcommand{\mod}{\mathop{\rm mod}}

\begin{document}

\markboth{\rm A.\ Langer}{\rm Lifting zero-dimensional schemes and divided powers}

\title{Lifting zero-dimensional schemes and divided powers}
\author{Adrian Langer}

\date{\today}

\maketitle

{\sc Address:}\\
Institute of Mathematics, University of Warsaw,
ul.\ Banacha 2, 02-097 Warszawa, Poland\\
e-mail: {\tt alan@mimuw.edu.pl}

\medskip

\begin{abstract}
  We study divided power structures on finitely generated
  $k$-algebras, where $k$ is a field of positive characteristic $p$.
  As an application we show examples of $0$-dimensional Gorenstein
  $k$-schemes that do not lift to a fixed noetherian local ring of
  non-equal characteristic. We also show that Frobenius neighbourhoods
  of a singular point of a general hypersurface of large dimension
  have no liftings to mildly ramified rings of non-equal
  characteristic.
\end{abstract}

{2010 \emph{Mathematics Subject Classification.} Primary 13D10; Secondary 13C40, 14B07}

\let\thefootnote\relax\footnote{Author's work was partially supported by
Polish National Science Centre (NCN) contract number 2015/17/B/ST1/02634.}

\section*{Introduction}

It is well-known that smooth projective curves defined over an
algebraically closed field of positive characteristic can be lifted to
characteristic zero. This is no longer the case for higher dimensional
projective varieties. However, every smooth scheme defined in positive
characteristic can be locally lifted to characteristic zero. Such
local lifting properties hold also for all locally complete
intersections as affine complete intersections are unobstructed.
Unfortunately, the lifting property does not hold for affine schemes
and in fact R. Vakil in \cite[Theorem 1.1, M7]{Va} shows that the
versal deformation spaces of isolated normal Cohen--Macaulay threefold
singularities satisfy Murphy's law, i.e., every singularity type of
finite type over $\ZZ$ appears on these spaces. This result depends on
Schlessinger's theorem, which says that under some mild assumptions
every deformation of a cone over a normal projective variety $X$ of
dimension $\ge 2$ is a cone over a deformation of $X$. If we take as $X$
a smooth projective surface that does not lift to characteristic zero,
then we obtain examples of singularities that do not lift to
characteristic zero. Clearly, this method does not work for lower
dimensional singularities and hence one could still hope for the local
lifting property for low dimensional schemes satisfying some nice
properties like being Cohen-Macaulay or Gorenstein. In fact, it is a
well-known problem whether there exist nonliftable zero-dimensional
schemes or nonliftable singular curves (see, e.g., \cite[p.~148]{Ha}
or \cite[Problem 1.2]{Am}). Although we cannot answer this question in
full generality, we show the following theorem (see Corollary
\ref{main-corollary}):

\begin{theorem}\label{unliftable-theorem}
  Let $R$ be a noetherian local ring with residue field $k$ of characteristic
  $p$. If $pR\ne 0$ then there exist $0$-dimensional Gorenstein
  $k$-schemes that are not liftable to $R$.
\end{theorem}

The constructed schemes depend on $R$ (more precisely, they depend
only on the smallest $e$ such that $m_R^e\subset pR$). So in principle
these schemes could be liftable to characteristic zero over some more
ramified rings but we are unable to check whether this really happens.
As a substitute we can find a direct system $\{ X_n\} _{n\in \NN}$ of
$0$-dimensional $k$-schemes such that for every noetherian local ring $R$ with
residue field $k$ and $pR\ne 0$ the schemes $X_n$ do not lift to $R$
for large $n$ (see Corollary \ref{direct-system}).

Let $k$ be an algebraically closed field of characteristic $p>0$.
If $X$ is a $k$-variety, the \emph{Frobenius neighbourhood} of a $k$-point $x\in X$
is the subscheme $(\Fr _X)^{-1}(x)\subset X$, where $\Fr_X:X\to X$ is the absolute Frobenius morphism. Set-theoretically it is equal to $x$ but its ideal sheaf in $\cO_X$ 
equals to $m_x^p\cdot \cO_X$.  For $r\ge 1$ an \emph{$r$th Frobenius neighbourhood} of $x\in X$ is defined similarly with  the Frobenius morphism $\Fr_X$ being replaced by $(\Fr_X)^r$.

We construct our examples by linkage from high Frobenius
neighbourhoods of a singular point of a hypersurface. Unfortunately,
deformations of Frobenius neighbourhoods of a vertex of a cone
over a projective variety $X$ are difficult to control and they
are not easily related to deformations of $X$. Still these
neighbourhoods seem to be ``less liftable'' than the original
variety, so it is an interesting question if they give unliftable
schemes for examples considered by Vakil.

The basic tools that we use are elementary deformation theory and
divided power algebra. In the simplest case of rings with small ramification (e.g., $R=W_2(k)$) 
existence of lifting is related to the study of divided power structures on 
an ideal of a ring of positive characteristic and we show a simple
criterion that allows us to check its existence. In this case one can use  Koblitz's  example  
\cite[Example 3.2.4]{BO} to get a non-liftable example (see below for more details). Lifting to other rings is more complicated and although it is not directly related to existence of divided power structures we can still give a numerical criterion that in some cases allows us to check that  lifting does not exist (see Theorem \ref{lift-criterion2}).

We show a criterion allowing us to check when higher Frobenius neighbourhoods
of singular points lift to some local rings with $pR\ne 0$ (see Theorem
\ref{lift-criterion2}). As a corollary we show the following
theorem (see Corollary \ref{generic-PD-structure} and Corollary
\ref{general-lifting}):

\begin{theorem}\label{main-generic}
Let $k$ be an algebraically closed field of characteristic $p>0$.
\begin{enumerate}
\item
The first Frobenius neighbourhood of a singular point of a general
hypersurface in $\AA ^n$, where $n\ge 6$, does not have a divided
power structure and it does not lift to $W_2(k)$.
\item Let
$X\subset \AA ^n_k$ be a general hypersurface with multiplicity
$\ge q=p^r$ at $0$. If $n\ge 3q$ then the $r$-th Frobenius
neighbourhood of $0\in X$ is not liftable to any local ring $R$
with residue field $k$ and such that $pR\ne 0$ and $m_R^{q}=0$.
\end{enumerate}
\end{theorem}

The only previous results related to Theorems
\ref{unliftable-theorem} and \ref{main-generic} are folklore.
Namely, it was known that if $k$ is perfect then there exists a
zero-dimensional $k$-scheme which does not lift to ring $W_2(k)$
of Witt vectors of length at most $2$. In fact, in \cite[Example
3.2.4]{BO} the authors state (without proof) N. Koblitz's example
of an ideal $J$ in a characteristic $p$ ring  with $J^{(p)}=0$
and with no divided power structure. The corresponding ring 
was known to have no lifting to $W_2 (k)$. The author learnt this
fact from B. Bhatt, who learnt it from J. de Jong. The proof was
published in \cite[Proposition 3.4]{Zd} by the author's student,
M. Zdanowicz (who learnt the fact from the author).
In characteristic $2$ we show that although this $0$-dimensional $k$-scheme 
does not lift to $W_2(k)$, it lifts to a discrete valuation ring of characteristic 
zero with absolute ramification $2$ (see Example \ref{Koblitz}). 
So in general one cannot expect that schemes from Theorem \ref{main-generic}
do not lift to characteristic zero. This explains why Frobenius neighbourhoods of Vakil's examples seem more likely to produce non-liftable examples of $0$-dimensional schemes.

\medskip

The structure of the article is as follows. In Section 1 we recall and
state a few preliminary results. In Section 2 we study divided power
structure on ideals close to Frobenius neighbourhoods. Then in Section
3 we prove the main technical criterion that allows to check
liftability of zero-dimensional rings. In Section 4 we apply these
results to obtain $0$-dimensional schemes that are not liftable to a
fixed ring. In Section 5 we show how to change these examples to
obtain Gorenstein schemes.

\section{Preliminaries}

\subsection{Simple lifting results}

Let us recall the following well-known lemma (see, e.g.,  \cite[Corollary to Theorem 22.5]{Ma}):

\begin{lemma} \label{Matsumura}
  Let $R\to S$ be a flat and local ring homomorphism of
  noetherian local rings. Denote by $m$ the maximal ideal of
  $R$. Let us assume that $f_1,...,f_s$ is a sequence of elements of $S$ such
  that their images form a regular sequence in $S/mS$. Then
  $f_1,...,f_s$ is a regular sequence in $S$ and the quotient
  $S/(f_1,...,f_s)$ is flat over $R$.
\end{lemma}

A sequence $f_1,...,f_s$ of elements of some ring $T$ determines a
$T$-linear homomorphism $t: T^s\to T$ (or equivalently a section $t\in
(T^s)^*=T^s$).  Let us recall that $t$ is called a \emph{regular
  section} if for all $i>0$ the homology groups $H_i(t)$ of the Koszul
complex of $t$ vanish. A sequence $f_1,...,f_s$ determines a regular
section if and only if $(f_1,...,f_s)$ is a regular sequence in $T_P$
for each prime ideal $P$ of $T$ that contains $f_1,...,f_s$.
Let us also recall that if  $(f_1,...,f_s)$ is a regular sequence in $T$
then  $f_1,...,f_s$ determines a regular section (see \cite[Theorem 16.5]{Ma}).

\begin{corollary} \label{regular-reduction}
  Let $R$ be a local Artin ring with a maximal ideal $m$ and
  let $\varphi : R\to S$ be a flat homomorphism of noetherian rings. Assume that
  $(f_1,...,f_s)$ is a sequence of elements of $S$ such that their
  images $(\bar f_1,...,\bar f_s)$ in $T=S/mS$ form a regular section of
  $T^s$.  Then $S/(f_1,...,f_s)$ is flat over $R$.
\end{corollary}

\begin{proof}
  The canonical projection $S\to S/(f_1,...,f_s)$ is denoted by $\psi$.
Let us recall  that $S/(f_1,...,f_s)$ is flat over $R$ if and only if for every
  prime ideal $P$ of $S/(f_1,...,f_s)$, the localization $(S/(f_1,...,f_s))_P$ is flat
  over $R_Q$, where $Q=(\psi \varphi)^{-1}(P)$.

Let us fix prime  ideal $P$ as above.
Since $R$ is local and Artin, every prime ideal in $R$ is equal to $m$
and hence $\Spec k\to \Spec R$ is a bijection. Since $T=S/mS\simeq
S\otimes _R k$, the canonical projection $\pi: S\to T$ also induces a
bijection $\Spec T\to \Spec S$.  In particular, there exists a prime
ideal $\bar P$ in $T$ such that $P':=\pi^{-1}(\bar P)= \psi^{-1}(P)=P+(f_1,...,f_s)$. By
construction $\bar P$ contains $\bar f_1,...,\bar f_s$.  Since $(\bar
f_1,...,\bar f_s)$ determines a regular section of $T^s$, the sequence
$(\bar f_1,...,\bar f_s)$ is regular in $T_{\bar P}$. Then by
Lemma \ref{Matsumura} $f_1,...,f_s$ is a regular sequence in $S_{P'}$ and
the quotient $S_{P'}/(f_1,...,f_s)=(S/(f_1,...,f_s))_{P}$ is flat over
$R_Q=R$.
\end{proof}

\medskip

Let us consider
$$0\to I\to \tilde R\to R\to 0,$$
where  $ \tilde R$ and $R$  are local Artin rings with residue field $k$ and ideal $I$
satisfies $m_{ \tilde R}I=0$.
The following lemma is contained in the proof of \cite[Theorem 10.1, p.80]{Ha}:

\begin{lemma} \label{Hartshorne}
  Let $A$ be a finitely generated $R$-algebra and let us assume that
  there exists a flat lifting of $A$ to $ \tilde R\to R$.  Let us choose a
  presentation of $A$ as a quotient $R[x_1,...,x_n]/(f_1,...,f_s)$.
  Then there exist elements $ \tilde f_1,..., \tilde f_s$ in
  $ \tilde R [x_1,...,x_n]$ lifting $f_1,...,f_s$ and such that
$ \tilde A= \tilde R[x_1,...,x_n]/(\tilde f_1,...,\tilde f_s)$ is a flat lifting of
  $A$ to $ \tilde R\to R$.
\end{lemma}

\medskip

Let $k$ be a field and let $R\to k$ be a surjective morphism from a
ring $R$.  We say that a $k$-scheme $X$ is \emph{liftable
  to $R\to k$} if there exists a flat $R$-scheme $\tilde X$ and a
closed embedding $X\hookrightarrow \tilde X$ inducing an isomorphism
$X\to \tilde X\times _R k$. We say that a $k$-algebra $A$ is s
\emph{liftable to $R\to k$} if the corresponding $k$-scheme $\Spec A$
is {liftable to $R\to k$}.

\medskip

\subsection{Divided power algebra}

Let $(A, I)$ be a commutative ring and an ideal. Let us recall that a
\emph{divided power structure on $I$} is a sequence of maps $\gamma
_n:I \to A$ that behave like operations $x\to x^n/n!$. More precisely, this sequence is required to satisfy the following properties for all $n,m \ge 0$, $x, y\in I$ and $a\in A$:
\begin{enumerate}
\item $\gamma_0(x)=1$, $\gamma _1(x)=x$ and $\gamma _n (x)\in I$,
\item $\gamma _n(x+y)=\sum _{i+j=n}\gamma_i(x)\gamma_j(y)$,
\item $\gamma_n (ax)=a^n\gamma_n(x)$,
\item $\gamma_n(x)\gamma_m(x)=\binom {n+m}{n} \gamma_{n+m}(x)$,
\item $\gamma_n(\gamma_m(x))=\frac{(nm)!}{n!(m!)^n}\gamma_{nm}(x).$
\end{enumerate}

For basic properties of divided power structures see \cite[\S 3]{BO} and \cite[Tag
09PD]{Stacks}.
The following lemma can be found in \cite[Tag 09PD, Lemma
5.3]{Stacks}. $\ZZ_{(p)}$ in the lemma stands for the ring of $p$-adic integers, i.e., the
localization of $\ZZ$ along the multiplicative system $\ZZ-(p)$. We avoid notation $\ZZ _p$
as in algebraic geometry this could be confused with the localization of $\ZZ$ along the multiplicative system $\{p^n\}_{n\ge 0}$.

\begin{lemma} \label{equiv-P.D.-maps} Let $p$ be a prime number and
  let $A$ be a $\ZZ_{(p)}$-algebra with an ideal $I$.  Then we have a
  natural bijection between the set of divided power structures
  $\gamma$ on $I$ and maps $\delta: I\to I$ such that
\begin{enumerate}
\item $p!\, \delta (x)=x^p$ for all $x\in I$,
\item $\delta (ax)=a^p\delta (x)$ for all $a\in A$ and $x\in I$,

\item $\delta (x+y)=\delta (x)+\delta (y)+\sum
  _{i=1}^{p-1}\frac{1}{i!(p-i)!}x^iy^{p-i}$ for all $x,y \in I$.
\end{enumerate}
The correspondence is given by $\delta =\gamma _p$.
\end{lemma}

\begin{corollary} \label{P.D.-hom}
  Let $p$ be a prime number and let $A$ be a $\ZZ_{(p)}$-algebra.  Let
  $(A, I, \gamma)$ and $(B, J, \eta)$ be divided power rings and let
  $\varphi: A\to B$ be a homomorphism of rings such that $\varphi
  (I)\subset J$. Let $T\subset I$ be a set of generators of $I$.  Then
  $\varphi$ induces a homomorphism of divided power rings $(A, I,
  \gamma)\to (B, J, \eta)$ if and only if $\eta _p (\varphi
  (t))=\varphi (\gamma _p (t))$ for all $t\in T$.
\end{corollary}

The above corollary follows easily from the definition of divided power structures together with the fact that
for $\ZZ_{(p)}$-algebras $\gamma_n$ is determined by $\gamma _p$ (which is the content of Lemma 
\ref{equiv-P.D.-maps}).

\medskip

Let $M$ be an $A$-module. Let us recall that a map $\tau : M\to M$
is called \emph{$p$-linear} if it is additive and $\tau
(ax)=a^p\tau (x)$ for all $a\in A$ and $x\in M$. For an ideal $I$ in
$A$ we denote by $I^{(p)}$ the ideal in $A$ generated by $x^p$ for all
$ x\in I$.

\begin{corollary}\label{PD-pI=0}
  Let $p$ be a prime number and let $A$ be a $\ZZ_{(p)}$-algebra with
  an ideal $I$.  Let us assume that $pI=0$ and $I$ admits a divided
  power structure.  Then $I^{(p)}=0$ and the set of all divided power
  structures on $I$ forms a torsor over the group of $p$-linear maps
  $\tau: I\to I$.  In particular, if $pA=0$ then the set of divided
  power structures on $(A,I)$ is isomorphic to $\Hom _A(I/I^2, F_*I)$.
\end{corollary}

\begin{proof}
  Vanishing of $I^{(p)}$ follows from condition 1 in Lemma
  \ref{equiv-P.D.-maps}.  If $\gamma$ and $\gamma\, {}'$ are divided power
  structures on $I$ then the difference $\tau=\gamma_p-{\gamma _p}'$ is
  $p$-linear. Lemma \ref{equiv-P.D.-maps} implies that if $\gamma$ is
  a divided power structure on $I$ and $\tau: I\to I$ is a $p$-linear
  map then there exists a unique divided power structure $\gamma\, {}' $ on
  $I$ such that ${\gamma_p}'=\gamma_p+\tau$. The second assertion
  follows from the fact that if $pA=0$ then $p$-linear maps $I\to I$
  correspond to $A$-linear maps $\tau: I\to F_*I$ and $\tau (I^2)=0$ as $I^{(p)}=0$.
\end{proof}

\subsection{Combinatorics}

Let us recall the following easy facts.  The following lemma can be found, e.g., in \cite[Theorem 1]{Fi}.

\begin{lemma}\emph{(Lucas's theorem)}\label{Lucas}
Let $p$ be a prime number and let $m$ and $n$ are non-negative integers.
Let us write $m=\sum a_i p^i$ and $n=\sum b_i p^i$, with $0\le a_i, b_i<p$.
Then
\[
\binom{m}{n}\equiv \prod \binom {a_i}{b_i} \mod p,
\]
where $\binom{a}{b}=0$ if $b>a$.
\end{lemma}

The next lemma is even more standard and it can be found in any book containing
combinatorial formulas:

\begin{lemma}\emph{(Vandermonde's identity)}\label{Vandermonde}
Let $l, m, n$ be non-negative integers.  Then we have
\[
\binom{m+n}{l}=\sum _{i=0}^l \binom{m}{i}\binom{n}{l-i}.
\]
\end{lemma}

\section{Locally complete intersections and divided powers}

Let $k$ be a field of characteristic $p$. Let us set
$S_0=k[x_1,...,x_n]$ and $m_{0}=(x_1,...,x_n)$ in $S_0$.  Since
$m_{0}^{(p)}\ne 0$ Corollary \ref{PD-pI=0} implies that $(S_0, m_{0})$ does
not have a divided power structure.  However, we have the following
lemma.

\begin{lemma}\label{P.D.-lci}
  Let $(R, I, \gamma)$ be a divided power ring. Let us assume that
  $(p-1)!$ is invertible in $R$ and $pI=0$.  Let us set
  $S=R[x_1,...,x_n]$, $I_S=IS+(x_1,...,x_n)$ and
  $A=S/(x_1^{i_1},...,x_n^{i_n})$, where $i_j\le p$ for all $j$. Then
  for any $y_1,...,y_n\in I_SA$ there exists a unique divided power
  structure $\tilde \gamma$ on $I_S A$ such that $\tilde \gamma _p
  (x_i)=y_i$ for $i=1,...,n$ and the canonical map $(R, I, \gamma)\to
  (A, I_SA, \tilde \gamma)$ is a homomorphism of divided power rings.
\end{lemma}

\begin{proof}
Let $R\langle x_1,...,x_n\rangle$ be the divided power polynomial
algebra (see \cite[Tag 09PD, Lemma 5.1]{Stacks}). This is an $R$-algebra with an $R$-module structure given by
$$R\langle x_1,...,x_n\rangle =\bigoplus _{m_1,...,m_n\ge 0}Rx_1^{[m_1]}...x_n^{[m_n]}$$
and multiplication given by
$$x_1^{[m_1]}...x_n^{[m_n]}\cdot x_1^{[m_1']}...x_n^{[m_n']}=
\prod _{i=1}^n\binom{m_i+m_i'}{m_i}
x_1^{[m_1+m_1']}...x_n^{[m_n+m_n']}. \leqno{(*)}$$ In particular, as an $R$-algebra ring $R\langle
x_1,...,x_n\rangle$ is generated  by $x_i^{[m]}$,
where $m\ge 0$ and $i=1,...,n$.

Let $R\langle x_1,...,x_n\rangle _{+}$ be the kernel of the
canonical map $R\langle x_1,...,x_n\rangle \to R$ sending
$x_i^{m}$ to zero for $m>0$. Let us set $J=IR\langle
x_1,...,x_n\rangle +R\langle x_1,...,x_n\rangle _{+}$. Then there
exists a unique divided power structure $\delta $ on $J$
such that $(R, I, \gamma)\to (R\langle x_1,...,x_n\rangle , J, \delta) $ is a homomorphism of divided power
rings and $\delta _m(x_i^{[1]})=x_i^{[m]}$ for all $m\ge 0$ and $i=1,..,n$.

Let us define a surjective homomorphism of $R$-modules $\varphi:
R\langle x_1,...,x_n\rangle \to A$ by sending
$x_1^{[m_1]}...x_n^{[m_n]}$ to
$\frac{1}{m_1!...m_n!}x_1^{m_1}...x_n^{m_n}$ if $m_j<i_j$ for all
$j=1,...,n$ and to $0$ otherwise. $(*)$ implies that this map
is a homomorphism of $R$-algebras. The kernel of $\varphi$ is an
ideal generated by $x_j^{[m]}$ for $m\ge i_j$ and $j=1,...,n$.
Since
$$\delta _l (x_i^{[m]})=\delta _l(\delta _m (x_i^{[1]}))=\frac{(lm)!}{l! (m!)^l}\delta _{lm} (x_i^{[1]})
=\frac{(lm)!}{l! (m!)^l}x_i^{[lm]},$$
\cite[Lemma 3.6]{BO} implies that the kernel of $\varphi$ is a sub-D.P.
ideal of $J$. Therefore by \cite[Lemma 3.5]{BO} we have an
induced divided power structure $\tilde \gamma^0$ on $I_SA$. Since
$\delta _p (x_i^{[1]})=x_i^{[p]}\in \ker \varphi$, we have $\tilde \gamma^0
_p (x_i)=0$ for $i=1,...,n$. Now for any  $y_1,...,y_n\in I_SA$ there exists
a $p$-linear map $\tilde \delta$ such that  $\tilde \delta _p (x_i)=y_i$
for $i=1,...,n$. Then $\tilde \gamma =\tilde \gamma ^0+\tilde \delta$ is
the required divided power structure on $I_SA$. Uniqueness of $\tilde \gamma $
follows from Lemma \ref{equiv-P.D.-maps}.
\end{proof}

\begin{example}
  Let $k$ be a field of characteristic $p>0$ and assume
  $i_1,...,i_n\le p$.  The above lemma fails for a general lifting $A$
  of $k[x_1,..,x_n]/(x_1^{i_1},...,x_n^{i_n})$ to $R\to k$.  This is
  clear if $pR=0$. For example if $A$ is a lifting of $A_0=k[x]/(x^p)$
  to $R\to k$ then we can write $A=R[x]/(x^p+g)$ for some $g\in m_RA$
  and any such ring is a lifting of $A_0$ to $R\to k$. But existence
  of a divided power structure on any ideal in $A$ containing $x$
  implies that $x^p=0$ in $A$ which is usually not the case.

  The lemma fails also for liftings if $pR\ne 0$.  For example let us
  take as $(R, I, \gamma)$ the ring $W_2(k)$ with $I=(p)$ and $\gamma
  _p=0$.  Let us consider $A=W_2(k)[x]/(x^p-p)$.  This ring is a flat
  lifting of $k[x]/(x^p)$ to $R\to k$. Let us assume that there exists
  a divided power structure $\tilde \gamma$ on ideal $(p,x)$.  Then we
  have $ p!\tilde \gamma _p (x)=x^p=p$, which implies $p((p-1)!\tilde
  \gamma _p (x)-1)=0$.  But $W_2(k)$-flatness of $A$ implies that
  $pA\simeq A/(p)$ and hence $((p-1)!\tilde \gamma _p (x)-1)\in (p)$.
  Since $\tilde \gamma _p (x)\in (p,x)$, this gives $1\in (p,x)$, a
  contradiction.
\end{example}

  Let $k$  be a field of characteristic $p$.  Let us set
  $A=k[x_1,...,x_n]/(x_1^{i_1},...,x_n^{i_n})$, where $2\le i_j\le p$ for all $j$.
  Lemma \ref{P.D.-lci} implies that there exists a unique divided power structure $\gamma$ on
$I_A=(x_1,...,x_n)A$ such that $\gamma _p (x_i)=0$ for
$i=1,...,n$. Let us consider $B=k[y_{1,1},...,y_{1,i_1-1}, ...,
y_{n,1},...,y_{n, i_n-1}]/(y_{i,j}^2)_{i=1,...,n, j=1,...,i_n-1}$.
This ring also comes with a unique divided power structure
$\delta$ on $I_B=(y_{l,j})_{l=1,...,n, j=1,...,i_l-1}B$ such that
$\delta _p (y_{i,j})=0$ for all $i$, $j$.

\begin{proposition} \label{DP-hom-example}
The map $x_l\to \sum _{j=1}^{i_l-1}y_{l,j}$ for $l=1,..,n$ defines
an injective homomorphism of  divided power rings $\varphi: (A,
I_A, \gamma)\to (B, I_B, \delta)$.
\end{proposition}

\begin{proof}
Let us take  an integer $2\le m<p$ and note that if $S_m$ acts on
$C_m=k[z_1,...,z_m]/(z_1^2,...,z_m^2)$ by permutation of variables
then the subring of invariants $C_m^{S_m}$ is spanned by
$(z_1+...+z_m)$. Moreover, this subring is isomorphic to
$k[t]/(t^{m+1})$ with isomorphism given by mapping $t$ to
$(z_1+...+z_m)$. To see this one needs to note that in $C_m$ we
have
$$(z_1+...+z_m)^s=s! \sum_{\substack{{l_1+...+l_m = s} \\ l_1,...,l_m\le 1}} z_1^{l_1}...z_m^{l_m} ,$$
which is zero precisely for $s>m$.

Now we have an action of $S_{i_1-1}\times ...\times S_{i_n-1}$ on
$B\simeq C_{i_1-1}\otimes ...\otimes C_{i_n-1}$ and the ring of
invariants is isomorphic  to $A\simeq C_{i_1-1}^{S_{i_1-1}}\otimes
...\otimes C_{i_n-1}^{S_{i_n-1}}$ with the isomorphism induced by
$\varphi: A\to B$. Clearly, we have $\varphi (I_A)\subset I_B$. By
Corollary \ref{P.D.-hom} to check that $\varphi$ is a homomorphism
of divided power rings it is sufficient to check that $\delta _p
(\varphi (x_l))=\varphi(\gamma _p(x_l))=0$ for $l=1,...,n$. But
since $i_l\le p$ we have
$$\delta _p (\varphi(x_l))= \delta _p (\sum _{j=1}^{i_l-1}y_{l,j})=\sum _{j=1}^{i_l-1}\delta _p(y_{l,j})+
\sum_{\substack{{s_1+...+s_{i_l-1} = p} \\ s_1,...,s_{i_l-1}\le
1}} y_{l,1}^{s_1}...y_{l,i_l-1}^{s_{i_l-1}}=0.$$
\end{proof}

\section{General results on divided power rings in equi-characteristic
  case}

Let $R$ be a ring in which $(p-1)!$ is invertible and let $f\in
R[x_1,...,x_n]$ be a polynomial. Let us write $f$ as a sum
$\sum_{i=1}^m a_{i}x^{J_i}$ of distinct monomials (where $J_i$ are multi-indices). Then we set
$$w_p(f):=\sum _{\sum_{j=1}^m l_j=p,\, l_j<p}\frac{1}{l_1!...l_m!}
(a_{1}x^{J_1})^{l_1}... (a_{m}x^{J_m})^{l_m} .$$ This polynomial
 appears naturally in  the computation of $f^p$ and it plays an important
role in the study of divided power structures.

\begin{proposition}\label{D.P.-criterion}
  Let $k$ be a ring of characteristic $p>0$. Let $m_{0}$ be the maximal ideal $(x_1,...,x_n)$ in  $k[x_1,...,x_n]$ and and let us
  take an ideal  $I\subset m_{0}^2$. Let us consider
  $A_0=k[x_1,...,x_n]/((x_1^{i_1},...,x_n^{i_n})+I)$, where $i_j\le
  p$ for all $j$.  Then the following conditions are equivalent:
\begin{enumerate}
\item $m_{0}A_0$ admits a divided power structure,
\item $ w_p(f_0) \in (x_1^{i_1},...,x_n^{i_n})+I$ for all $f_0\in
I$,
\item if $I$ is generated by some subset $T_0$ then
$ w_p(f_0) \in (x_1^{i_1},...,x_n^{i_n})+I$ for  all $f_0\in T_0$.
\end{enumerate}
\end{proposition}

\begin{proof}
  Let us take some $f_0\in I$ and
  write $f_0=\sum_{i=1}^m a_{i}x^{J_i}$, where $J_i$ are
  distinct multiindices with $|J_i|\ge 2$.

Assume that  $m_{0}A_0$ has a divided power structure $\gamma$.
Note that $\gamma _p (x_ix_j)=x_i^p\gamma _p (x_j)=0$. Similarly, since $|J_i|\ge 2$
we get $\gamma _p (x^{J_i})=0$.
Then
$$0=\gamma _p(f_0)=\sum \gamma _p (a_i x^{J_i}) +w_p(f_0)=\sum a_i^p\gamma _p ( x^{J_i}) +w_p(f_0)=w_p(f_0)$$
in $A_0$. Hence $w_p(f_0)\in (x_1^{i_1},...,x_n^{i_n})+I$, which
proves that 1 implies 2. Obviously 2 implies 3. To prove that 3
implies 1 let us set $B_0=k[x_1,...,x_n]/(x_1^{i_1},...,x_n^{i_n})$
and assume that $w_p(f_0)\in (x_1^{i_1},...,x_n^{i_n})+I$ for all
$f_0\in T$.  By Lemma \ref{P.D.-lci} the ideal $m_0B_0$ has a unique
divided power structure $\tilde \gamma$ such that $\tilde \gamma
(x_i)=0$ and the canonical map $(k, 0, 0)\to (B_0, m_0B_0, \tilde
\gamma)$ is a homomorphism of divided power rings. By the same
computation as above we have
$$\tilde \gamma _p(f_0)=\sum \tilde \gamma _p (a_i x^{J_i}) +w_p(f_0)=w_p(f_0)\in IB_0$$
and hence \cite[Lemma 3.6]{BO} implies that $IB_0$ is a sub-D.P.
ideal of $m_0B$. Therefore by \cite[Lemma 3.5]{BO} there exists a
unique divided power structure $\gamma$ on $m_0A_0$ such that
$(B_0, m_0B_0, \tilde \gamma)\to (A_0, m_0A_0, \gamma)$ is a
homomorphism of divided power algebras.
\end{proof}

\begin{corollary}
  Let $k$ be a ring of characteristic $p>0$. Let
  $I \subset (x_1,...,x_n)^2\subset  k[x_1,...,x_n]$ and $J\subset (y_1,...,y_{n'})^2\subset k[y_1,...,y_{n'}]$ be ideals and let $\pi_x: k[x_1,...,x_n, y_1,...,y_{n'}]
  \to k[x_1,...,x_n]$ and $\pi_y: k[x_1,...,x_n, y_1,...,y_{n'}]\to k[ y_1,...,y_{n'}]$ be the canonical projections.
  Let $i_1,...,i_n, j_1,...,j _{n'}$ be positive integers less or equal to $p$.
If the ideal $(x_1,...,x_n,y_1,...,y_{n'})$ in
  $k[x_1,...,x_n, y_1,...,y_{n'}]/((x_1^{i_1},...,x_n^{i_n}, y_1^{j_1},...,
  y_{n'}^{j_{n'}})+ \pi_x^{-1}I+\pi_y^{-1}J)$
admits a divided power structure then
$(x_1,...,x_n)$ in $k[x_1,...,x_n]/((x_1^{i_1},...,x_n^{i_n})+ I)$
and
$(y_1,...,y_{n'})$ in $k[y_1,...,y_{n'}]/((y_1^{j_1},...,y_n^{j_{n'}})+J)$
also admit a divided power structure.
\end{corollary}

\begin{proof}
Let us take some polynomials $f(x)\in I$ and $g(y)\in J$.
The canonical projection  $\pi_x$ maps $w_p(f(x)+g(y))$ to $w_p (f(x))$
and  $\pi_y$ maps $w_p(f(x)+g(y))$ to $w_p (g(y))$, so the
corollary follows directly from Proposition \ref{D.P.-criterion}.
\end{proof}

\medskip

  \begin{theorem} \label{lift-criterion2} 
Let $(R, m_R)$ be a local  ring with residue field $k$ of characteristic $p>0$..  
Let us assume that $pR\ne 0$, $pm_R=0$ and $m_R^{e+1}=0$ for some $1\le e\le q-1$.
Let us take an ideal $I\subset k[x_1,...,x_n]$ and some polynomial $f_0\in I$ such that for each multiindex
    $(l_1,...,l_n)$ of a monomial occurring in $f_0$ we have
$$\left\lfloor \frac{ql_1}{i_1}\right\rfloor+...+\left\lfloor \frac{ql_n}{i_n}\right\rfloor \ge e+1.$$
Let us set $A_0=k[x_1,...,x_n]/((x_1^{i_1},...,x_n^{i_n})+I)$, where $i_j\le q=p^r$ for all $j$.
If $A_0$ is liftable to the canonical projection $R\to k=R/m_R$ then
\[w_p(f_0^{p^{r-1}})\in (x_1^{i_1},...,x_n^{i_n})+I.\]
\end{theorem}

\begin{proof}
  Assume that $A_0$ is liftable to $R\to k$ and let $\pi: A\to A/m_RA=
  A_0$ denote the corresponding projection. By Lemma \ref{Hartshorne}
  we can assume that $A$ is a quotient of $R[x_1,...,x_n]$ and $\pi$
  lifts to a map $R[x_1,...,x_n]\to k[x_1,...,x_n]$.  Let us write
  $f_0$ as a sum of monomials $\sum_{i=1}^m a_{i}x^{J_i}$.  Let us choose some
  $b_{i}\in R$ lifting $a_{i}\in k$ and let us set $f=\sum_{i=1}^m
  b_{i}x^{J_i}$. By construction we have $\pi (f)=f_0=0$ in $A_0$, so
  $f\in m_RA$.  This implies $f^q\in m_R^qA=0$. Similarly, we have
  $x_j^{i_j}\in m_R A$ as $x_j^{i_j}=0$ in $A_0$. Then our assumptions
  on $e$ and $J_i$ imply that $x^{qJ_i}\in m_R^{e+1}A=0$ for
  $i=1,...,m$.  Hence computing in $A$ we get
\[
0=f^q=\left(\sum_{i=1}^m b_{i}x^{J_i}\right)^q=(\sum_{i=1}^m
b_{i}^px^{pJ_i}+p!w_p(f))^{p^{r-1}}.
\]
If $r=1$ then we have
\[
0=(\sum_{i=1}^m b_{i}^px^{pJ_i}+p!w_p(f))^{p^{r-1}}=p!w_p(f).
\]
Note that $pR\subset m_R$ and $pm_R=0$, so $p^2R=0$. So if $r>1$ then we obtain
\[
0=(\sum_{i=1}^m b_{i}^px^{pJ_i}+p!w_p(f))^{p^{r-1}}=(\sum_{i=1}^m b_{i}^px^{pJ_i})^{p^{r-1}}.
\]
By induction replacing $q=p^r$ by $p^{r-1}$ and $\sum_{i=1}^m b_{i}x^{J_i}$ by $\sum_{i=1}^m b_{i}^px^{pJ_i}$
we eventually get
\[
p!w_p(f(x_1^{p^{r-1}} ,..., x_n^{p^{r-1}}))=0
\]
Note that $p:R\to R$ factors through $k=R/m_R\to R$ and this last map
is injective as $pR\ne 0$. Therefore from $R$-flatness of $A$, the map
$\tau: A_0\simeq k\otimes _R{A} \to {A}$ is also injective and by
construction $\tau\pi=p$.  But we have
\[ \tau ((p-1)!\,  w_p(f_0(x_1^{p^{r-1}} ,..., x_n^{p^{r-1}}))) =
\tau \pi ((p-1)!\,w_p(f(x_1^{p^{r-1}} ,..., x_n^{p^{r-1}})) = p! w_p(f(x_1^{p^{r-1}} ,..., x_n^{p^{r-1}}))=0.
\]
So we have $w_p(f_0^{p^{r-1}})=w_p(f_0(x_1^{p^{r-1}} ,..., x_n^{p^{r-1}}))=0$ in $A_0$ and hence
$w_p(f_0^{p^{r-1}})\in (x_1^{i_1},...,x_n^{i_n})+I.$
\end{proof}

\medskip

\begin{remark}
In the special case when $I$ is a principal ideal and
$i_1=...=i_n=q$ the above theorem has the following geometric
reformulation. Let $X:=(f_0=0)\subset \AA_k^n$ be a hypersurface
with multiplicity $\ge (e+1)$ at $0$. If $q=p^r$ and the $r$-th Frobenius
neighbourhood $Y_r:=\Spec k[x_1,...,x_n]/(x_1^q,...,x_n^q,f_0)$ of
$0$ is liftable to any local ring $R\to k$ then $w_p(f_0^{p^{r-1}})=0\in k[Y_r].$
\end{remark}

\medskip

\begin{corollary} \label{lift-criterion} Let $(R, m_R)$ be a noetherian local
  ring with $pR\ne 0$ and residue field $k$.  Let us assume that
  $m_R^{e+1}=0$ for some $1\le e\le p-1$.  Let $T_0$ be a set
  generating some ideal $I\subset k[x_1,...,x_n]$ and let us assume that for each $f_0\in T_0$ and each
  multiindex $(l_1,...,l_n)$ of a monomial occurring in $f_0$ we have
$$\left\lfloor \frac{pl_1}{i_1}\right\rfloor+...+\left\lfloor \frac{pl_n}{i_n}\right\rfloor \ge e+1.$$
Let us set $A_0=k[x_1,...,x_n]/((x_1^{i_1},...,x_n^{i_n})+I)$, where $i_j\le p$ for all $j$.
If $A_0$ is liftable to $R\to k$ then $m_0A_0$
has a divided power structure.
\end{corollary}

\begin{proof} 
Replacing $R$ by $R/pm_R$ and using Nakayama's lemma we can assume that $pm_R=0$. 
Then the required assertion follows from Proposition \ref{D.P.-criterion} and the above theorem.
\end{proof}

\medskip

Note that if $A_0$ is liftable to $R\to k$ then the proof of Proposition \ref{D.P.-criterion} gives  
the same divided power structure on $m_0A_0$ independently of a lifting.

\medskip

Let $V$ be a discrete valuation ring of unequal characteristic $p$ and with
uniformizing parameter $\pi$. Let us assume that $(\pi)$ has a divided
power structure.  By \cite[Example 3.2.3]{BO} this is equivalent to
$e\le p-1$, where $e$ is the absolute ramification index of $V$. Then
$R=V/(\pi^{e+1})$ is a local ring satisfying $pR=(\pi ^e)\ne 0$ and
$m_R^{e+1}=0$.

Let $k$ be a field of characteristic $p>0$ and let $R$ be some ring
with surjection $R\to k$. Assume that some $k$-algebra $A_0$ is
liftable to $R\to k$. Since flatness is preserved under localization,
$A_0$ is liftable to the localization $R'=R_{(p-1)!}\to k$. If $(R,
m_R)$ is as in Corollary \ref{lift-criterion} then $(p-1)!$ is
invertible in $R'$ and $(m_RR')^p=0$, so $m_RR'$ has a not necessarily unique
divided power structure (cf. \cite[Example 3.2.4]{BO}).

In the following we show examples of rings $A_0$ as in Corollary
\ref{lift-criterion} such that $m_0A_0$ does not have a divided
power structure. However, in some examples of Corollary
\ref{DP-existence} one can see that $T^1_{A_0/k}$ is large, i.e.,
$A_0$ has many non-trivial deformations over the ring $k[t]/(t^2)$
of dual numbers. Since $(t)$ in $k[t]/(t^2)$ has a divided power
structure, this show that assumption $pR\ne 0$ in Corollary
\ref{lift-criterion} is essential.

\section{Examples}

In this section $k$ is a field of characteristic $p>0$.

\begin{lemma}\label{checking-lemma}
Let $q=p^r$, where $r\ge 1$.
  Let $f_0=y_1+y_2+y_3\in k[x_1,...,x_n]$ be a sum of non-zero
  monomials, each of total degree $\ge 2$ and such that each variable
  $x_i$ appears in product $y_1 ^{q-p^{r-1}}y_2^{q-p^{r-1}}y_3^{2p^{r-1}-1}$ with degree less than $q$.  Then
\[w_p(f_0^{p^{r-1}})\not \in (x_1^{q},...,x_n^{q}, f_0).\]
\end{lemma}

\begin{proof}
Let us assume that $w_p (f_0^{p^{r-1}})\in (x_1^{q},...,x_n^{q},f_0)$. It follows that $f_0^{q-1}w_p
  (f_0^{p^{r-1}})\in (x_1^{q},...,x_n^{q})$.  Let us note that
\[ (p-1)!f_0^{q-1}w_p (f_0^{p^{r-1}})= \sum_{j_1+j_2+j_3 = q-1}
\binom{q-1}{j_1,\,j_2,\,j_3}
  \sum_{\substack{{l_1+l_2 +l_3= p} \\ l_1,l_2,l_3<p}} \frac{(p-1)!}{l_1!l_2!l_3!}
y_1^{j_1+{p^{r-1}}l_1}y_2^{j_2+{p^{r-1}}l_2}y_3^{j_3+{p^{r-1}}l_3}.
\]
If this polynomial belongs to the ideal $(x_1^{q},...,x_n^{q})$ then the
coefficient $\alpha$ at the monomial $y_1 ^{q-p^{r-1}}y_2^{q-p^{r-1}}y_3^{2p^{r-1}-1}$ is $0$.
This coefficient is equal to
\[ \alpha= \sum_{\substack{{l_1+l_2+l_3 = p} \\ l_1,l_2, l_3 < p}} \frac{(p-1)!}{l_1!l_2!l_3!}
  \sum_{\substack{{j_1+j_2 +j_3= q-1} \\ p^{r-1}l_1+j_1=q-p^{r-1},p^{r-1} l_2+j_2=q-p^{r-1}, p^{r-1} l_3+j_3=2p^{r-1}-1}}
\binom{q-1}{j_1,\,j_2,\,j_3}.
\]
Let us note that $j_1= p^{r-1}(p-1-l_1)$ and $j_2= p^{r-1}(p-1-l_2)$
are $p$-adic expansions.  Since $p^{r-1} l_3+j_3=2p^{r-1}-1$ we have
$l_3\le 1$ and $j_3=(p-1)+...+(p-1)p^{r-2}+(1-l_3)p^{r-1}$ is a
$p$-adic expansion of $j_3$. Note also that
$q-1=(p-1)+...+(p-1)p^{r-1}$. So by Lucas's theorem (see Lemma
\ref{Lucas}) we have
\[
\binom{q-1}{j_1,\,j_2,\,j_3}=\binom{q-1}{j_3}\binom{q-1-j_3}{j_1}\equiv \binom{p-1}{1-l_3}\binom{p-2+l_3}{p-1-l_1} \mod p.
\]
Using Vandermonde's identity (see Lemma \ref{Vandermonde}) we get
\begin{align*}
&  \alpha= \sum_{l_1+l_2= p-1} \frac{(p-1)!}{l_1!l_2!} \binom{p-1}{0} \binom{p-1}{p-1-l_1}+
  \sum_{\substack{{l_1+l_2 = p} \\ l_1,l_2 < p}} \frac{(p-1)!}{l_1!l_2!}
 \binom{p-1}{1}\binom{p-2}{p-1-l_1}\\
     &=  \sum_{l_1=0}^{p-1} \binom{p-1}{l_1}\binom{p-1}{p-1-l_1}+
  \sum_{l_1 =1}^{p-1} \frac{(p-1)!}{l_1!(p-l_1)!} \frac{(p-1)!}{(p-1-l_1)! (l_1-1)!}\\
&={\binom{2p-2}{p-1}}+ \sum_{l_1=0}^{p} \binom{p-1}{l_1}\binom{p-1}{p-l_1}
={\binom{2p-2}{p-1}}+{\binom{2p-2}{p}}=\binom{2p-1}{p}=\binom{(p-1)+p}{p}=1,
 \end{align*}
where in the last line we again use Lucas's theorem. This contradicts our assumption.
\end{proof}

\begin{corollary}\label{DP-existence}
  Let $A_0=k[x_1,...,x_n]/(x_1^{p},...,x_n^{p}, f_0)$, where
  $f_0=y_1+y_2+y_3$ is a sum of non-zero monomials, each of total
  degree $\ge 2$ and such that each variable $x_i$ appears in product
  $y_1(y_2y_3)^{p-1}$ with degree less than $p$.  Then the ideal
  $(x_1,...,x_n)\subset A_0$ does not have a divided power structure.
\end{corollary}

\begin{proof}
 If $(x_1,...,x_n)\subset A_0$ has a divided power
  structure then Proposition \ref{D.P.-criterion} implies that $w_p
  (f_0)\in (x_1^{p},...,x_n^{p},f_0)$. But this contradicts Lemma \ref{checking-lemma}.
\end{proof}

\begin{example}\label{Koblitz}
As a special case of the above corollary we obtain Koblitz's example
\cite[Example 3.2.4]{BO}: the ideal $(x_1,...,x_6)$ in
$A_0=k[x_1,...,x_6]/(x_1^p,..., x_6^p, x_1x_2+x_3x_4+x_5x_6)$ does not
have a divided power structure.  In this case Corollary
\ref{lift-criterion} implies that $A_0$ does not lift to $W_2(k)$
recovering \cite[Proposition 3.4]{Zd}.  
Below we show that if $k$ has characteristic $2$ then  this ring 
lifts to characteristic zero.
For simplicity we take $k=\FF_2$ although 
the same construction works for an arbitrary field of characteristic $2$.

Let us consider $R=\ZZ [\sqrt{2}]=\ZZ [t]/(t^2-2)$.
This ring has a canonical surjection $R\to R\otimes \FF_2=\FF_2[t]/(t^2)\to \FF_2$.
Let us set 
$$A=R[x_1,...,x_6]/(x_1^2+tx_4x_5x_6, x_2^2+tx_3, x_3^2,x_4^2,x_5^2, x_6^2, x_1x_2+x_3x_4+x_5x_6).$$
Note that we divide by a non-homogeneous ideal, so unlike $A_0$ ring $A$ does not have a canonical grading (but it has a weighted grading, e.g., we can assign $(x_1,...,x_6)$ weights
$(3,1,2,2,2,2)$).
Using any computer algebra system one can check that $A\otimes _{R}\QQ (\sqrt{2})$ is a $\QQ(\sqrt{2})$-algebra of length $36$. One can also check that $A\otimes _{R}\FF_2=A_0$
is an $\FF_2$-algebra of the same length. Therefore after localization we see that $A\left[ \frac{1}{2}\right]$ is a flat $R\left[ \frac{1}{2}\right]$-module lifting $A_0$ to characteristic zero. In fact, with some more work one can probably check that $A$ is a flat $R$-module but we will not need that.

This gives the first known example of a $0$-dimensional scheme defined over a field $k$ of positive characteristic that does not lift to $W_2(k)$ but it lifts to characteristic zero.
\end{example}

\medskip

\begin{corollary}\label{generic-PD-structure}
Let $k$ be an algebraically closed field of characteristic $p>0$.
Let $X\subset \AA ^n_k$ be a general hypersurface singular at $0$.
Let us assume that $n\ge 5$ if $p\ge 3$ or $n\ge 6$ if $p=2$. Then
the first Frobenius neighbourhood of $0\in X$ has no divided power
structure and it does not lift to $W_2(k)$.
\end{corollary}

\begin{proof}
Let us choose coordinates $x_1,...,x_n$ in $\AA^n_k$ and let $m_0$
be the maximal ideal $(x_1,...,x_n)\subset k[\AA^n_k]=k[x_1,...,x_n]$.
Let $f=0$ be an equation of a hypersurface singular at $0$. 
Note
that $M(f):=(p-1)!f^{p-1}w_p(f)$ is an integer polynomial in
coefficients of $f$ and hence the condition $M(f)\in
(x_1^p,...,x_n^p)$ defines a closed subset in the space
$m_0^2/(x_1^p,...,x_n^p)$ parameterizing the first Frobenius
neighbourhoods of hypersurfaces singular at $0$. Clearly, this subset 
does not correspond to all hypersurfaces and it is non-obvious that it is 
non-empty.  But the proof of
Lemma \ref{checking-lemma} shows that there exists $f$ for which
$M(f)\not \in (x_1^p,...,x_n^p)$. More precisely, one can take
$f=x_1^2+x_2x_3+x_4x_5$ if $p\ge 3$ or $f=x_1x_2+x_3x_4+x_5x_6$ if
$p=2$. So a general hypersurface also satisfies this condition and
by Proposition \ref{D.P.-criterion} its first Frobenius
neighbourhood has no divided power structure. By Corollary
\ref{lift-criterion} such schemes do not lift to
$W_2(k)$.
\end{proof}

\begin{remark}
In the above corollary the notion of ``general''  is used in 
the usual sense, i.e., it corresponds to a general point
in the parameter space of all hypersurfaces singular at $0$.
However, the proof shows that the assertion holds also for a cone over 
a general projective hypersurface of degree $2$ in $\PP ^{n-1}$ (under 
the same assumptions on $n$). One can also obtain a similar statement 
for hypersurfaces of higher degree at the cost of increasing the number 
of variables and degree of hypersurfaces (see the proof of Corollary 
\ref{general-lifting}).
\end{remark}

\medskip

The next proposition gives for any local ring $R$ an example of   a $0$-dimensional
scheme $Z\subset \AA ^{n}_k$ given by only quadratic equations but
non-liftable to $R$.

\begin{proposition}
  Let $(R, m_R)$ be a local ring with $pR\ne 0$ and residue
  field $k=R/m_R$ of characteristic $p$. Let us assume that
  $m_R^{p-1}=0$.  Let us take $n=6(p-1)$ and consider the ring
$C=k[y_{1,1},...,y_{6,p-1}]/(y_{1,1}^{2},...,y_{6,p-1}^{2}, g)$,
where
$$g=\sum_{i=1}^{p-1}\sum_{j=1}^{p-1} (y_{1,i}y_{2,j}+y_{3,i}y_{4,j}+y_{5,i}y_{6,j}).$$
Then $Z:=\Spec C$ is not liftable to $R\to k$.
\end{proposition}

\begin{proof}
By Proposition \ref{DP-hom-example} we have an injective
homomorphism of divided power rings
\[ \varphi: (A=k[x_1,...,x_6]/(x_1^p,..., x_6^p),
I_A, \gamma)\to
(B=k[y_{1,1},...,y_{6,p-1}]/(y_{1,1}^{2},...,y_{6,p-1}^{2}), I_B,
\delta)
\] given by $x_l\to \sum _{j=1}^{p-1}y_{l,j}$. Let us set
$f=x_1x_2+x_3x_4+x_5x_6$. Then $\varphi (f)=g$ and hence $\varphi
(w_p(f))=w_p(g)$. By Lemma \ref{DP-existence} and Proposition
\ref{D.P.-criterion} we have $w_p(f)\not \in fA$. Using the group
action as in proof of Proposition \ref{DP-hom-example} one can
easily see that this implies that $w_p(g)\not \in gB$ and hence
the ideal $(y_{i,j})_{i=1,...,6, j=1,...,p-1}$ in $C=B/gB$ does
not have a divided power structure. Hence the required assertion
follows from Corollary \ref{lift-criterion}.
\end{proof}

\medskip

\begin{proposition}\label{non-liftable}
  Let $(R, m_R)$ be a local ring with $pR\ne 0$ and residue
  field $k=R/m_R$ of characteristic $p$. Let us assume that
  $m_R^{q}=0$ for some $q=p^r$ with $r\ge 1$. 
Let us take $n=3q$ and consider the ring $A_0=k[x_1,...,x_n]/(x_1^{q},...,x_n^{q}, f_0)$, where
$$f_0=x_1x_2...x_q+x_{q+1}x_{q+2}... x_{2q}+x_{2q+1}x_{2q+2}...x_{3q}.$$  
Then $A_0$ is not liftable to $R\to k$.
\end{proposition}

\begin{proof}
As in proof of Corollary \ref{lift-criterion} we can assume that $pm_R=0$.
Then by Theorem \ref{lift-criterion2} it is sufficient to show that
  $w_p(f_0^{p^{r-1}})\not \in (x_1^{q},...,x_n^{q}, f_0)$. This is a
  direct corollary of Lemma \ref{checking-lemma}.
\end{proof}

\begin{corollary} \label{general-lifting}
Let $k$ be an algebraically closed field of characteristic $p>0$
and let $(R, m_R)$ be a local ring with $pR\ne 0$ and residue
field $k$. Let us assume that  $m_R^{q}=0$, where $q=p^r$.  Let
$X\subset \AA ^n_k$ be a general hypersurface with multiplicity
$\ge q$ at $0$. If $n\ge 3q$ then the $r$-th Frobenius
neighbourhood of $0\in X$ is not liftable to $R\to k$.
\end{corollary}

\begin{proof}
The proof is analogous to the proof of Corollary
\ref{generic-PD-structure}, with Proposition \ref{non-liftable}
giving an example of a polynomial $f$ for which $f^{q-1}w_p
(f^{p^{r-1}})\not \in (x_1^{q},...,x_n^{q})$.
\end{proof}

\begin{corollary} \label{direct-system} Let $k$ be a field of
  characteristic $p>0$.  There exists a direct system $\{ X_n\} _{n\in
    \NN}$ of $0$-dimensional $k$-schemes such that for any noetherian local ring
  $(R, m_R)$ with $pR\ne 0$ and residue field $k=R/m_R$ the schemes $X_n$
do not lift to $R\to k$ for all sufficiently large $n$.
\end{corollary}

\begin{proof}
  We construct the required system inductively starting with
  $X_0=\Spec k$.  Suppose that we constructed $X_n$ and it is of the
  form $\Spec k[x_1,...,x_{m_n}]/(x_1^{i_1},...,x_{m_n}^{i_{m_n}},
  f)$.  Let us set
\[ A_{n+1}: = k[x_1,...,x_{m_n}, y_{1},...,
y_{3p^n}]/(x_1^{i_1},...,x_{m_n}^{i_{m_n}}, y_{1}^{p^n},...,
y_{3p^n}^{p^n}, f(x)+ g(y)),
\]
where  $g(y)= y_1...y_{p^n}+y_{p^n+1}... y_{2p^n}+y_{2p^n+1}...y_{3p^n})$.
 By construction we have a surjective homomorphism
$$\varphi_{n+1}: A_{n+1}\to B_{n+1}:=k[y_{1},...,y_{3p^n}]/(y_{1}^{p^n},..., y_{3p^n}^{p^n}, g(y)), $$
which maps $w_p((f(x)+g(y))^{p^{n-1}})$ to
$w_p(g(y))^{p^{n-1}})$. Since $w_p(g(y))^{p^{n-1}})\ne 0$ in
$B_{n+1}$, we also have $w_p((f(x)+g(y))^{p^{n-1}})\ne 0$ in
$A_{n+1}$. So by Theorem \ref{lift-criterion2}  if we  set $X_{n+1}:=\Spec A_{n+1}$ then
$X_{n+1}$ does not lift to local rings $(R, m_R)$ with $m_R ^{p^n}=0$ and $pR\ne 0$.

Now if $(R, m_R)$ is any noetherian local ring with $pR\ne 0$ and
residue field $k=R/m_R$ then by Krull's intersection theorem we can
find some $e$ such that $m_R^{e+1}\subset pR$. If $e\le p^n$ then
$X_{n+1}$ does not lift to $R/m_R^{e}$ and hence it also does not lift
to $R\to k$.
\end{proof}

\section{Liftability of $0$-dimensional Gorenstein schemes }

Let $X\subset Z$ be a subscheme. Define a sheaf of ideals $I_Y =
\cHom_{\cO_Z} (\cO_X , \cO_Z )$, and let $Y\subset Z$ be the subscheme
defined by $I_Y$. Then we say that \emph{$Y$ is linked to $X$ by $Z$}.

\medskip

Let us set $S=k[x_1,\ldots,x_{n}]$ and $B_0= S/
{(x_1^p,\ldots,x_{n}^p)}$. Let us fix some $f\in S$ and set
$A_0=B_0/fB_0$.
Let us consider a local Artin ring $R$ with  residue field $k$.
The second part of the next proposition is a special case of a variant of \cite[Exercise 9.4]{Ha}.

\begin{proposition}\label{non-liftable-Gorenstein}
  The scheme $Y_0=\Spec C_0$, where $C_0=B_0/(0:_{B_0}f)$, is a $0$-dimensional
  Gorenstein $k$-scheme.  If $Y_0$ has a lifting to $R\to k$ then $X_0=\Spec A_0$ also
  has a lifting to $R\to k$.
\end{proposition}

\begin{proof}
We have an exact sequence of $B_0$-modules
$$0\longrightarrow J\longrightarrow B_0\stackrel{\cdot f}{\longrightarrow} B_0\longrightarrow B_0/fB_0\longrightarrow 0,$$
where $J=(0:_{B_0}f)$. Since $B_0$ is $0$-dimensional and Gorenstein, we
have $(0:_{B_0}J)=fB_0\simeq B_0/J$. Therefore $\omega_{B_0/J}=\Hom _{B_0} (B_0/J,
B_0)\simeq B_0/J$ and $C_0=B_0/J$ is Gorenstein.

Let us set $Z_0=\Spec B_0$ and let $P_0$ be the spectrum of the
localization of $S$ at the maximal ideal $(x_0,...,x_n)$. Then $Z_0$
is a $0$-dimensional complete intersection $k$-subscheme of $P_0$
and $Y _0\subset Z_0$ is a Gorenstein $k$-subscheme linked to $X_0$ by
$Z_0$.

Assume that $Y_0$ is liftable to $R\to k$. By Lemma \ref{Hartshorne}
there exists a lifting $Y\subset P$ of $Y_0\subset P_0$
to $R\to k$, where $P$ is the localization of
$R[x_0,..,x_n]$ at the maximal ideal lying over $(x_0,...,x_n)$.
Then there exists a lifting $Z$ of $Z_0$ to $R\to k$ that
contains $Y$ (see \cite[Exercise 9.4]{Ha}).  In fact, by
Lemma \ref{Matsumura} one can take as $Z$ a subscheme of
$P$ cut out by some lifts of generators of the ideal of $Z_0$
in $P_0$ taken from the ideal of $Y$ in $P$. But then
the scheme $X$, linked to $Y$ by $Z$, is a
lifting of $X_0$ to $R\to k$.
\end{proof}

\medskip

\begin{corollary}\label{main-corollary}
  Let $R$ be a noetherian local ring with residue field $k$ of characteristic
  $p>0$. If $pR\ne 0$ then there exists a zero-dimensional Gorenstein
  $k$-scheme $Z$ that cannot be lifted to $R$.
\end{corollary}

\begin{proof}
  Let $m_R$ be the maximal ideal of $R$. By Krull's intersection
  theorem $\bigcap m_R^n=0$, so we can find the smallest positive
  integer $e$ such that $m_R^{e}\subset pR$. If a $k$-scheme is
  liftable to $R$ then it is also liftable to $R'=R/pm_R$. Let us set
  $m_{R'}=m_RR'$. Then $pm_{R'}=0$ and $m_{R'}^{e}\subset pR'$, so
  $m_{R'}^{e+1}=0$. Moreover, we have $pR'\ne 0$. Indeed, if $pR'=0$
  then $pR=pm_R$, so by Nakayama's lemma $pR=0$, a contradiction.  Now
  the required assertion follows from Propositions \ref{non-liftable}
  and \ref{non-liftable-Gorenstein}.
\end{proof}

\subsection*{Acknowledgements}

The author would like to thank P. Achinger, J. Jelisiejew and M. Zdanowicz for
useful comments on earlier versions of the paper.

\end{document}